\documentclass[a4paper]{amsart} 
\usepackage{amssymb} 
\usepackage{tikz-cd}
\usepackage{hyperref}
\calclayout

\title{Deriving Auslander's formula}
\author{Henning Krause}
\address{Henning Krause\\ Fakult\"at f\"ur Mathematik\\
  Universit\"at Bielefeld\\ D-33501 Bielefeld\\ Germany.}
\email{hkrause@math.uni-bielefeld.de}

\newtheorem{lemma}{Lemma}[section]
\newtheorem{proposition}[lemma]{Proposition}
\newtheorem{corollary}[lemma]{Corollary}
\newtheorem{theorem}[lemma]{Theorem}

\newtheorem*{Thm}{Theorem}

\theoremstyle{remark}
\newtheorem{remark}[lemma]{Remark}

\theoremstyle{definition}
\newtheorem{example}[lemma]{Example}
\newtheorem{definition}[lemma]{Definition}

\numberwithin{equation}{section}

\renewcommand{\mod}{\operatorname{mod}\nolimits}
\DeclareMathOperator*{\colim}{colim}

\newcommand{\Proj}{\operatorname{Proj}\nolimits}
\newcommand{\proj}{\operatorname{proj}\nolimits}

\newcommand{\card}{\operatorname{card}\nolimits}

\newcommand{\Add}{\operatorname{Add}\nolimits}

\newcommand{\Inj}{\operatorname{Inj}\nolimits}

\newcommand{\Lex}{\operatorname{Lex}\nolimits}
\newcommand{\Id}{\operatorname{Id}\nolimits}
\newcommand{\Mod}{\operatorname{Mod}\nolimits}
\newcommand{\Mor}{\operatorname{Mor}\nolimits}
\newcommand{\Ind}{\operatorname{Ind}\nolimits}
\newcommand{\umod}{\operatorname{\underline{mod}}\nolimits}
\newcommand{\End}{\operatorname{End}\nolimits}
\newcommand{\Flat}{\operatorname{Flat}\nolimits}
\newcommand{\fp}{\operatorname{fp}\nolimits}

\newcommand{\Hom}{\operatorname{Hom}\nolimits}

\newcommand{\HOM}{\operatorname{\mathcal H \;\!\!\mathit o \mathit m}\nolimits}

\newcommand{\Ker}{\operatorname{Ker}\nolimits}
\newcommand{\Coker}{\operatorname{Coker}\nolimits}

\newcommand{\eff}{\operatorname{eff}\nolimits}
\newcommand{\Eff}{\operatorname{Eff}\nolimits}
\newcommand{\Ext}{\operatorname{Ext}\nolimits}

\newcommand{\Loc}{\operatorname{Loc}\nolimits}

\newcommand{\Thick}{\operatorname{Thick}\nolimits}

\newcommand{\Ab}{\mathsf{Ab}}
\newcommand{\op}{\mathrm{op}}

\newcommand{\comp}{\mathop{\circ}}
\newcommand{\lto}{\longrightarrow}
\newcommand{\xto}{\xrightarrow}

\def\a{\alpha}
\def\b{\beta}

\def\s{\sigma}

\def\la{\lambda}

\def\Ga{\Gamma}

\def\Si{\Sigma}

\def\A{{\mathsf A}}
\def\B{{\mathsf B}}
\def\C{{\mathsf C}}
\def\D{{\mathsf D}}

\def\Sc{{\mathsf S}}

\def\Y{{\mathsf Y}}

\def\T{{\mathsf T}}

\def\bfAc{\mathbf{Ac}}
\def\bfC{{\mathbf C}}
\def\bfD{{\mathbf D}}
\def\bfK{{\mathbf K}}

\def\bfR{{\mathbf R}}
\def\bfS{{\mathbf S}}

\def\bbZ{{\mathbb Z}}

\begin{document}

\begin{abstract}
  Auslander's formula shows that any abelian category $\C$ is
  equivalent to the category of coherent functors on $\C$ modulo the
  Serre subcategory of all effaceable functors. We establish a derived
  version of this equivalence. This amounts to showing that the
  homotopy category of injective objects of some appropriate
  Grothendieck abelian category (the category of ind-objects of $\C$)
  is compactly generated and that the full subcategory of compact
  objects is equivalent to the bounded derived category of $\C$. The
  same approach shows for an arbitrary Grothendieck abelian category
  that its derived category and the homotopy category of injective
  objects are well-generated triangulated categories. For sufficiently
  large cardinals $\a$ we identify their $\a$-compact objects and
  compare them.
\end{abstract}

\subjclass[2010]{18E30 (primary), 16E35, 18C35, 18E15}

\maketitle

\setcounter{tocdepth}{1}
\tableofcontents

\section{Introduction}

Let $\A$ be a Grothendieck abelian category and let $\Inj\A$ denote
the full subcategory of injective objects. Then it is known from
Neeman's work \cite{Ne2001b,Ne2014} that the derived category
$\bfD(\A)$ and the homotopy category $\bfK(\Inj\A)$ are well-generated
triangulated categories. The present work describes for sufficiently
large cardinals $\a$ their subcategories of $\a$-compact objects.

Recall that any well-generated triangulated category $\T$ admits a
filtration $\T=\bigcup_{\a}\T^\a$ where $\a$ runs through all regular
cardinals and $\T^\a$ denotes the full subcategory of $\a$-compact
objects \cite{Ne2001}. This is an analogue of the filtration
$\A=\bigcup_{\a}\A^\a$ where $\A^\a$ denotes the full subcategory of
$\a$-presentable objects \cite{GU1971}.  Note that there exists a
regular cardinal $\a_0$ such that $\A^\a$ is abelian for all
$\a\ge\a_0$ (Corollary~\ref{co:PG}). In fact, when $\A^\a$ is abelian
and generates $\A$, then $\A^\b$ is abelian for all $\b\ge\a$
(Corollary~\ref{co:abelian}).

We distinguish two cases, keeping in mind the notation
$\A^{\aleph_0}=\fp\A$ and $\T^{\aleph_0}=\T^c$. The first case is a
generalisation of the locally noetherian case studied in
\cite{Kr2005}.

\begin{Thm}[$\a=\aleph_0$]
  Let $\A$ be a Grothendieck abelian category. Suppose that the
  subcategory $\fp\A$ of finitely presented objects is abelian and
  generates $\A$. Then the homotopy category $\bfK(\Inj\A)$ is a
  compactly generated triangulated category and the canonical functor
  $\bfK(\Inj\A)\to\bfD(\A)$ induces an equivalence
  $\bfK(\Inj\A)^c\xto{\sim}\bfD^b(\fp\A)$.
\end{Thm}

\begin{Thm}[$\a\neq\aleph_0$]
  Let $\A$ be a Grothendieck abelian category. Suppose that
  $\a>\aleph_0$ is a regular cardinal such that the subcategory
  $\A^\a$ is abelian and generates $\A$. Then the following holds:
\begin{enumerate}
\item The  derived category $\bfD(\A)$ is
an $\a$-compactly generated triangulated category and the inclusion
$\A^\a\to\A$ induces an equivalence $\bfD(\A^\a)\xto{\sim}\bfD(\A)^\a$.
\item The homotopy category $\bfK(\Inj\A)$ is an $\a$-compactly
  generated triangulated category and the left adjoint of the
  inclusion $\bfK(\Inj\A)\to\bfK(\A)$ induces a quotient functor
  $\bfK(\A^\a)\twoheadrightarrow\bfK(\Inj\A)^\a$.
\end{enumerate}
\end{Thm}

The case $\a=\aleph_0$ is Theorem~\ref{th:KInj} and
for $\a\neq\aleph_0$ see Theorems~\ref{th:derived} and
\ref{th:KInj2}.

Note that in case $\a=\aleph_0$ the derived category $\bfD(\A)$ need
not be compactly generated; an explicit example is given by Neeman
\cite {Ne2014}.

The above results are obtained by `resolving' the abelian category
$\A$.  More precisely, we use a variation of Auslander's formula
(Theorem~\ref{th:auslander}) to write $\A$ as the quotient of a
functor category modulo an appropriate subcategory of effaceable
functors (Corollary~\ref{co:abelian}). Then we `derive' this
presentation of $\A$ by passing to the derived category $\bfD(\A)$ and
to the homotopy category $\bfK(\Inj\A)$.  This passage from a
Grothendieck abelian category to a well-generated triangulated
category demonstrates the amazing parallel between both concepts
\cite{Kr2004}. Also, we see the relevance of the filtration
$\A=\bigcup_{\a}\A^\a$, which seems to be somewhat neglected in the
literature.

There are at least two aspects that motivate our work. The homotopy
category $\bfK(\Inj\A)$ played an import role in work with Benson and
Iyengar on modular representations of finite groups
\cite{BIK2011}. For instance, a classification of localising
subcategories of $\bfK(\Inj\A)$ amounts to a classification of
$\a$-localising subcategories of $\bfK(\Inj\A)^\a$ for a sufficiently
large cardinal $\a$. On the other hand, $\bfK(\Inj\A)$ has been used
to reformulate Grothendieck duality for noetherian schemes, and it
seems reasonable to wonder about the non-noetherian case; see
\cite{Ne2014} for details.

This paper has two parts. The first sections form the `finite' part,
dealing with finitely presented and compact objects. Cardinals greater
than $\aleph_0$ only appear in the last section, which includes a
gentle introduction to locally presentable abelian and well-generated
triangulated categories.

\section{Functor categories and Auslander's formula}

In this section we recall definitions and some basic facts about
functor categories. In particular, we recall Auslander's formula.

\subsection*{Localisation sequences}

We consider pairs of adjoint functors $(F,G)$
\[\begin{tikzcd}
\C \arrow[twoheadrightarrow,yshift=0.75ex]{rr}{F} &&\D  \arrow[tail,yshift=-0.75ex]{ll}{G}
\end{tikzcd}\]
satisfying the following equivalent conditions  \cite[I.1.3]{GZ}:
\begin{enumerate}
\item The functor $F$ induces an equivalence
\[\C[\Si^{-1}]\stackrel{\sim}\lto\D\]
where $\Si:=\{\s\in\Mor\C\mid F\s\text{ is invertible }\}$.
\item The functor $G$ is fully faithful. 
\item The morphism of functors $FG\to\Id_\D$ is invertible.
\end{enumerate}

\begin{definition}
  A \emph{localisation sequence} of abelian (triangulated) categories
  is a diagram of functors
\begin{equation}\label{eq:loc}
\begin{tikzcd}
\B \arrow[tail,yshift=0.75ex]{rr}{E} &&\C  \arrow[twoheadrightarrow,yshift=-0.75ex]{ll}{E'}
\arrow[twoheadrightarrow,yshift=0.75ex]{rr}{F} &&\D  \arrow[tail,yshift=-0.75ex]{ll}{F'}
\end{tikzcd}
\end{equation}
satisfying the following conditions:
\begin{enumerate}
\item  $E$ and $F$ are exact functors of  abelian (triangulated) categories.
\item The pairs $(E,E')$ and $(F,F')$ are adjoint pairs.
\item The functors $E$ and $F'$ are fully faithful.
\item An object in $\C$ is annihilated by $F$ iff it
  is in the essential image of $E$.
\end{enumerate}
\end{definition}
We refer to \cite{Ga1962,Ve1997} for basic properties, in particular
for the construction of the abelian (triangulated) quotient $\C/\B'$
such that $F$ induces an equivalence $\C/\B'\xto{\sim}\D$, where $\B'$
denotes the full subcategory of objects in $\C$ that are annihilated
by $F$. Thus any of the functors $E,E',F,F'$ determines the diagram
\eqref{eq:loc} up to equivalence.

An exact functor $F\colon \C\to\D$ of abelian (triangulated)
categories is by definition a \emph{quotient functor} if $F$ induces
an equivalence $\C/\B\xto{\sim}\D$, where $\B$ denotes the full
subcategory of objects in $\C$ that are annihilated by $F$.

\subsection*{Finitely presented functors}

Let $\C$ be an additive category. We denote by $\mod\C$ the
category of finitely presented functors $F\colon\C^\op\to\Ab$. Recall
that $F$ is \emph{finitely presented} (or \emph{coherent}) if it fits
into an exact sequence
\begin{equation}\label{eq:pres}
\Hom_\C(-,X)\lto \Hom_\C(-,Y)\lto F\lto 0.
\end{equation}
The \emph{Yoneda
functor}  is the fully faithful functor
\[\C\lto\mod\C,\quad X\mapsto\Hom_\C(-,X).\]

\subsection*{Additive functors}

Let $\C$ be an (essentially) small additive category. A
\emph{$\C$-module} is an additive functor $\C^\op\to\Ab$.  For the
category of $\C$-modules we write
\begin{align*}
\Mod\C:=&\Add(\C^\op,\Ab)\\
\intertext{and consider the following full
subcategories:}
\Proj\C:=&\text{ projective $\C$-modules}\\
\Inj\C:=&\text{ injective $\C$-modules}\\
\Flat\C:=&\text{ flat $\C$-modules}
\end{align*}
The flat $\C$-modules are precisely the filtered colimits of
representable functors. Thus we may identify
\[\Flat \C=\Ind\C\]
where $\Ind\C$ denotes the category of ind-objects in the sense of
\cite[\S 8]{Grothendieck/Verdier:1972a}. When $\C$ admits cokernels then
\[\Flat \C=\Lex(\C^\op,\Ab)\]
where $\Lex(\C^\op,\Ab)$ denotes the category of left exact functors
$\C^\op\to\Ab$.

Note that the inclusion $\mod\C\to\Mod\C$ induces an
equivalence\[\Ind\mod\C\stackrel{\sim}\lto\Mod\C.\]

\subsection*{Effaceable functors}
Let $\C$ be an abelian category.  Then we write $\eff\C$ for the
full subcategory of functors $F$ in $\mod\C$ that admit a presentation
\eqref{eq:pres} with $X\to Y$ an epimorphism in $\C$. If $\C$ is small
then the inclusion $\eff\C\to\Mod\C$ induces a fully faithful
functor
 \[\Eff\C:=\Ind\eff\C\lto\Mod\C.\] 
This identifies $\Eff\C$ with the functors in $\Mod\C$ that are
 effaceable in the sense of \cite[p.~141]{Gr1957}.

\subsection*{Auslander's formula}

The following result is somewhat hidden in Auslander's account on
coherent functors.

\begin{theorem}[{\cite[p.~205]{Au1966}}]\label{th:auslander}
Let $\C$ be an abelian category. Then the Yoneda
  functor $\C\to\mod\C$ induces a localisation sequence of abelian categories.
\begin{equation}\label{eq:aus}
\begin{tikzcd}
  \eff\C \arrow[tail,yshift=0.75ex]{rr} && \mod\C
  \arrow[twoheadrightarrow,yshift=-0.75ex]{ll}
  \arrow[twoheadrightarrow,yshift=0.75ex]{rr} &&\C
  \arrow[tail,yshift=-0.75ex]{ll}
\end{tikzcd}
\end{equation}
Moreover, the functor $\mod\C\to\C$ induces an equivalence
\[\frac{\mod\C}{\eff\C}\stackrel{\sim}\lto \C.\]
\end{theorem}
\begin{proof}
  The left adjoint of the Yoneda functor is the unique functor
  $\mod\C\to\C$ that preserves finite colimits and sends each
  representable functor $\Hom_\C(-,X)$ to $X$. Thus a functor $F$ with
  presentation \eqref{eq:pres} is sent to the cokernel of the
  representing morphism $X\to Y$. The exactness of the left adjoint
  follows from a simple application of the horseshoe lemma.
\end{proof}

Following Lenzing \cite{Le1998} we call this presentation of an
abelian category (via the left adjoint of the Yoneda functor)
\emph{Auslander's formula}.

A predecessor of this result for Grothendieck abelian categories is
due to Gabriel.

\begin{theorem}[{\cite[II.2]{Ga1962}}]\label{th:coh}
  Let $\C$ be a small abelian category. Then the inclusion
  $\Ind\C\to\Mod\C$ induces a localisation sequence of abelian
  categories extending \eqref{eq:aus}.
\begin{equation}\label{eq:aus2}
\begin{tikzcd}  
\Eff\C \arrow[tail,yshift=0.75ex]{rr} && \Mod\C
  \arrow[twoheadrightarrow,yshift=-0.75ex]{ll}
  \arrow[twoheadrightarrow,yshift=0.75ex]{rr} &&\Ind\C
  \arrow[tail,yshift=-0.75ex]{ll}
\end{tikzcd}
\end{equation} 
Moreover, the functor $\Mod\C\to\Ind\C$ induces an equivalence
\[\frac{\Mod\C}{\Eff\C}\stackrel{\sim}\lto \Ind\C.\]
\end{theorem}
\begin{proof}
  The left adjoint of the inclusion functor is the unique functor
  $\Mod\C=\Ind\mod\C\to\Ind\C$ that preserves filtered colimits and
  extends the functor $\mod\C\to\C$ from
  Theorem~\ref{th:auslander}. Any exact sequence in $\Mod\C$ can be
  written as a filtered colimit of exact sequences in
  $\mod\C$.\footnote{Let $\eta\colon 0\to X\xto{\a} Y\xto{\b} Z\to 0$
    be exact. Write $\a$ as filtered colimit of morphisms $\a_i\colon
    X_i\to Y_i$ in $\mod\C$. Each $\a_i$ induces an epimorphism
    $\b_i\colon Y_i\to \Coker\a_i$. Then $\eta$ is the filtered colimit
    of the exact sequences $0\to\Ker\b_i\to Y_i\to\Coker\a_i\to 0$.}
  Thus the exactness of the functor $\mod\C\to\C$ yields the exactness
  of $\Mod\C\to\Ind\C$.
\end{proof}

Note that for any small abelian category $\C$ the category $\Ind\C$ is
a Grothendieck abelian category \cite[II.3]{Ga1962}. Thus all categories
occuring in diagram \eqref{eq:aus2} are Grothendieck abelian.  

A Grothendieck abelian category $\A$ has injective envelopes and we
denote by $\Inj\A$ the full subcategory of injective objects. The
diagram \eqref{eq:aus2} induces a sequence of functors\footnote{The
  notation $\Inj\C$ is ambiguous: We mean the category of injective
  $\C$-modules, and not the category of injective objects in $\C$.}
$\Inj\Ind\C\to\Inj\C\to\Inj\Eff\C$ since a right adjoint of an exact
functor preserves injectivity.

\section{Derived categories}

In this section we describe a derived version of Auslander's formula
for the bounded derived category of an abelian category.

Let us recall some notation. For an additive category $\C$ let
$\bfK(\C)$ denote the homotopy category of cochain complexes in $\C$.
The objects of $\bfK(\C)$ are cochain complexes and the morphisms are
homotopy classes of chain maps.  When $\C$ is abelian, then
$\bfAc(\C)$ denotes the full subcategory of acyclic complexes and the
derived category $\bfD(\C)$ is by definition the triangulated quotient
$\bfK(\C)/\bfAc(\C)$. The superscript ${}^b$ refers to the full
subcategory of cochain complexes $X$ satisfying $X^n=0$ for $|n|\gg
0$.

For a triangulated category $\T$ and a class of objects $\Sc$ in $\T$
let $\Thick(\Sc)$ denote the smallest thick subcategory of $\T$
containing $\Sc$.

\begin{lemma}\label{le:derived}
  Let $\C$ be an abelian category.  Then the Yoneda functor
  $\C\to\mod\C$ induces a triangle equivalence
  $\bfK^b(\C)\xto{\sim}\bfD^b(\mod\C)$ that makes the following
  diagram commutative.
\[\begin{tikzcd}[row sep=scriptsize]
  \bfAc^b(\C) \arrow[tail]{rr}\arrow{d}{\wr}&& \bfK^b(\C)
  \arrow{d}{\wr} \arrow[twoheadrightarrow]{rr}&&
  \bfD^b(\C) \arrow[equal]{d}\\
  \Thick(\eff\C) \arrow[tail]{rr}&&\bfD^b(\mod\C)
  \arrow[twoheadrightarrow]{rr}&&\bfD^b(\C)
\end{tikzcd}\]
In particular, the triangulated category $\bfAc^b(\C)$ is generated by the acyclic
  complexes of the form \[\cdots\to 0\to X^{n-1}\to X^{n} \to X^{n+1} \to
  0\to\cdots.\]
\end{lemma}
\begin{proof}
  Each object $F$ in $\mod\C$ admits a finite projective resolution
\begin{equation}\label{eq:res}
0\to\Hom_\C(-,X)\to \Hom_\C(-,Y)\to
  \Hom_\C(-,Z)\to F\to 0.
\end{equation}
Thus we have a triangle equivalence
$\bfK^b(\C)\xto{\sim}\bfD^b(\mod\C)$ because the Yoneda functor
identifies $\C$ with the full subcategory of projective objects in
$\mod\C$.

For $X$ in $\bfK^b(\C)$ let $Y$ denote the corresponding complex in
$\bfD^b(\mod\C)$ and observe that $Y$ belongs to
$\Thick(\{H^i(Y)\mid i\in\bbZ \})$.  The kernel of $\mod\C\to\C$
equals $\eff\C$ by Theorem~\ref{th:auslander}. Thus $Y$ is in
$\Thick(\eff\C)$ iff $Y$ is annihilated by
$\bfD^b(\mod\C) \to\bfD^b(\C)$ iff $X$ is in $\bfAc^b(\C)$. This
yields the triangle equivalence $\bfAc^b(\C)\xto{\sim}\Thick(\eff\C)$.

For the final assertion of the lemma, observe that the equivalence
$\bfK^b(\C)\xto{\sim}\bfD^b(\mod\C)$ identifies the complexes of the
form
 \[\cdots\to 0\to X^{n-1}\to X^{n} \to X^{n+1} \to
 0\to\cdots\] with the objects in $\eff\C$ viewed as complexes
 concentrated in degree $n+1$.
\end{proof}

\begin{corollary}\label{co:derived}
  \pushQED{\qed} The canonical functor $\bfD^b(\mod\C)\to\bfD^b(\C)$
  induces an equivalence
\[\frac{\bfD^b(\mod\C)}{\Thick(\eff\C)}\stackrel{\sim}\lto\bfD^b(\C).\qedhere\]
\end{corollary}

Observe that $\Thick(\eff\C)$ identifies with the full subcategory of
complexes $X$ in $\bfD^b(\mod\C)$ such that $H^n(X)$ belongs to $\eff\C$
for all $n\in\bbZ$; see also Lemma~\ref{le:relative}.

I am grateful to Xiao-Wu Chen for pointing out the following.

\begin{remark}
  The inclusion $\eff\C\to\mod\C$ induces a functor
  $\bfD^b(\eff\C)\to\bfD^b(\mod\C)$ which is not fully faithful in
  general.  Examples arise by taking for $\C$ the category $\mod A$ of
  finite dimensional modules over a finite dimensional $k$-algebra
  $A$, where $k$ is any field. Then $\eff\C$ identifies with
  $\mod(\umod A)$, where $\umod A$ denotes the stable category modulo
  projectives \cite[\S 6]{AR1974}.
\end{remark}

\section{Homotopy categories of injectives}

In this section we discuss a derived version of Auslander's formula
for complexes of injective objects. More precisely, we extend the
presentation of a small abelian category $\C$ via Auslander's formula
to the homotopy category of injective objects of $\Ind\C$.

\subsection*{Homotopically minimal complexes}

Let $\A$ be a Grothendieck abelian category. Recall from
\cite[Proposition~B.2]{Kr2005} that every complex $I$ in $\A$ with
injective components admits a decomposition $I = I'\amalg I''$ such
that $I'$ is homotopically minimal and $I''$ is null homotopic. Here,
a complex $J$ is \emph{homotopically minimal} if for all $n$ the
inclusion $\Ker d_J^n\to J^n$ is an injective envelope.

Recall that a full subcategory $\B\subseteq\A$ is \emph{localising} if
$\B$ is closed under forming subobjects, quotients, extensions, and
coproducts.

\begin{lemma}\label{le:env}
  Let $\A$ be a Grothendieck abelian category and $\B$ a localising
  subcateory.  Write $t\colon\A\to\B$ for the right adjoint of the
  inclusion. If a complex $I$ with injective components in $\A$ is
  homotopically minimal, then $tI$ is homotopically minimal in $\B$.
\end{lemma} 
\begin{proof}
  Observe that $\Ker d_{tI}^n= t(\Ker d_I^n)$ since $t$ is left
  exact. Now use that $t$ takes injective envelopes in $\A$ to 
 injective envelopes in $\B$.
\end{proof}

\subsection*{Pure acyclic complexes}

Let $\A$ be a Grothendieck abelian category and fix a class $\C$ of
objects that generates $\A$ and is closed under finite
colimits. Throughout we identify objects in $\A$ with complexes
concentrated in degeee zero.

The following is a slight generalisation of a result due to
\v{S}\v{t}ov\'i\v{c}ek \cite{St2011}.

\begin{proposition}\label{pr:pac}
  Let $I$ be a complex of injective objects in $\A$ and suppose that
  $\Hom_{\bfK(\A)}(X,I[n])=0$ for all $X\in\C$ and $n\in\bbZ$. Then
  $I$ is null homotopic.
\end{proposition}

The proof is based on the following lemma, where $\bfC(\A)$ denotes
the abelian category of cochain complexes in $\A$.

\begin{lemma}[\cite{St2011}]\label{le:pac}
  Let $\Y\subseteq \bfK(\Inj\A)$ be a class of objects satisfying
  $\Y=\Y[1]$. Then \[^\perp \Y:=\{X\in\bfK(\A)\mid
  \Hom_{\bfK(\A)}(X,Y)=0\text{ for all }Y\in\Y\}\] has the following
  properties.
\begin{enumerate}
\item Let $0\to X'\to X\to X''\to 0$ be an exact sequence in
  $\bfC(\A)$.  If two terms are in $^\perp \Y$, then the third term
  belongs to $^\perp \Y$.
\item Let $X=\bigcup X_i$ be a directed union of subobjects
  $X_i\subseteq X$ in $\bfC(\A)$. If all $X_i$ belong to $^\perp \Y$,
  then $X$ belongs to $^\perp \Y$.
\end{enumerate}
\end{lemma}
\begin{proof}
(1) Let $Y\in\Y$. Then the induced sequence
\[0\to \HOM_\A(X'',Y)\to \HOM_\A(X,Y)\to \HOM_\A(X',Y)\to 0\] is
exact, where $\HOM_\A(-,-)$ denotes the usual Hom complex. Now observe that
\[\Hom_{\bfK(\A)}(-,Y[n])=H^n\HOM_\A(-,Y).\]

 (2) Let $Y\in\Y$ and observe that
 $\Hom_{\bfK(\A)}(-,Y[1])=\Ext^1_{\bfC(\A)}(-,Y)$. Thus the assertion
 follows from Eklof's lemma \cite{Ek1977}, using that it suffices to
 treat well-ordered chains; see \cite{St2011} for details.
\end{proof}

\begin{proof}[Proof of Proposition~\ref{pr:pac}]
  Let $\Y$ denote the class of complexes $I$ in $\bfK(\Inj\A)$ such
  that $\Hom_{\bfK(\A)}(X,I[n])=0$ for all $X\in\C$ and $n\in\bbZ$.
  We consider $^\perp \Y\subseteq\bfK(\A)$ and have $\C\subseteq{^\perp
    \Y}$ by definition. Now fix an object $X$ in $\C$ and a subobject
  $U\subseteq X$ in $\A$. Write $\bigcup U_i=U$ as a directed union
  of subobjects which are quotients of objects in $\C$. Thus each object
  $X/U_i$ belongs to $^\perp \Y$ since $\C$ is closed under
  cokernels. Now apply Lemma~\ref{le:pac}. Thus each $U_i$ is in
  $^\perp \Y$, therefore $U$, and finally $X/U$ belongs to $^\perp
  \Y$. Each object in $\A$ is a directed union of subobjects of the
  form $X/U$. Thus all objects of $\A$ belong to $^\perp\Y$. Clearly,
  this implies that all complexes in $\Y$ are null homotopic.
\end{proof}

Let us call a class $\C\subseteq\A$ \emph{saturated} if it is
closed under finite colimits and if there is a localising subcategory
$\B\subseteq\A$ such that $\B$ is generated by $\C$. 

\begin{example}[{\cite[Theorem~2.8]{Kr1997}}]
  Let $\A$ be a Grothendieck abelian category. Suppose that the full
  subcategory $\fp\A$ of finitely presented objects is abelian and
  that it generates $\A$. Then any Serre subcategory of $\fp\A$ is
  saturated.
\end{example}

The following is an immediate consequence of Proposition~\ref{pr:pac}.

\begin{proposition}\label{pr:min}
  Let $\A$ be a Grothendieck abelian category.  For a saturated class
  $\C$ of objects and a homotopically minimal complex of injective
  objects $I$, the following are equivalent:
\begin{enumerate} 
\item  $\Hom_{\bfK(\A)}(X,I[n])=0$ for all $X\in\C$ and $n\in\bbZ$.
\item $\Hom_\A(X,I^n)=0$ for all $X\in\C$ and $n\in\bbZ$.
\end{enumerate}
\end{proposition}
\begin{proof}
  (1) $\Rightarrow$ (2): Let $\B$ denote the localising subcategory of
  $\A$ generated by $\C$ and write $t\colon\A\to\B$ for the right
  adjoint of the inclusion. The assumption implies that
  $\Hom_{\bfK(\B)}(X,tI[n])=0$ for all $X\in\C$ and $n\in\bbZ$. Thus
  $tI$ is null homotopic by Proposition~\ref{pr:pac}, and therefore
  $tI=0$ by Lemma~\ref{le:env}.

(2) $\Rightarrow$ (1): Clear since $\Hom_{\bfK(\A)}(X,I[n])=H^n\HOM_\A(X,I)$.
\end{proof}

\subsection*{Compactly generated triangulated categories}

Let $\T$ be a triangulated category and suppose that $\T$ admits small
coproducts. An object $X$ in $\T$ is \emph{compact} if each morphism
$X\to\coprod_{i\in I}Y_i$ in $\T$ factors through $\coprod_{i\in
  J}Y_i$ for some finite subset $J\subseteq I$. Let $\T^c$ denote the
full subcategory of compact objects and observe that $\T^c$ is a thick
subcategory. Following \cite{Ne1996}, the triangulated category $\T$
is \emph{compactly generated} if $\T^c$ is essentially small (that is,
the isomorphism classes of objects form a set) and $\T$ admits no
proper localising subcategory containing $\T^c$.

\begin{example}
  For a small additive category $\C$, the derived category
  $\bfD(\Mod\C)$ is compactly generated with subcategory of compact
  objects given by $\bfK^b(\C)\xto{\sim}\bfD(\Mod\C)^c$.
\end{example}

We shall need the following well-known result about Bousfield
localisation for compactly generated triangulated categories.

\begin{proposition}\label{pr:local}
  Let $\T$ be a compactly generated triangulated category and
  $\Sc\subseteq\T^c$ a triangulated subcategory. Then the triangulated category
\[\Sc^\perp:=\{Y\in\T\mid
\Hom_\T(X,Y)=0\text{ for all }X\in\Sc\}\] has small coproducts and is
compactly generated. Moreover, the left adjoint of the inclusion
$\Sc^\perp\to \T$ induces (up to direct summands) an equivalence
$\T^c/\Sc\xto{\sim} (\Sc^\perp)^c$.
\end{proposition} 
\begin{proof}
Combine \cite[Theorem~2.1]{Ne1992} and \cite[Theorem~9.1.16]{Ne2001}.
\end{proof}

\subsection*{Homotopy categories of injectives}

We describe the homotopy category of injective objects for
Grothendieck abelian categories of the form $\Ind\C$ given by a small
abelian category $\C$.  The following lemma provides the basis; it is
the special case where $\Ind\C$ is replaced by $\Mod\C$.

\begin{lemma}\label{le:mod}
  Let $\C$ be a small abelian category. Then the canonical functor
  \[Q\colon\bfK(\Inj\C)\lto\bfD(\Mod\C)\] is a triangle equivalence which
  restricts to an equivalence
\begin{equation}\label{eq:mod}
\bfK(\Inj\C)^c\stackrel{\sim}\lto\bfD^b(\mod\C).
\end{equation}
\end{lemma}
\begin{proof}
  Recall from \cite{Sp1988} that the restriction of $Q$ to the full
  subcategory of K-injective complexes is a triangle
  equivalence. Moreover, each $X$ in $\bfK(\Mod\C)$ fits into an exact
  triangle $aX\to X\to iX\to$ with $aX$ acyclic and $iX$ K-injective.

  Each $F\in\mod\C$ admits a finite projective resolution
  \eqref{eq:res} since $\C$ is abelian. Then a standard argument
  yields $\Hom_{\bfK(\Mod\C)}(F,I[n])=0$ for all $n\in\bbZ$ and each
  acyclic complex $I$ of injectives, since this holds when $F$ is
  projective.  Thus $I$ is null homotopic by Proposition~\ref{pr:pac},
  and we conclude that $Q$ is an equivalence. It remains to observe
  that $\bfK^b(\C)\xto{\sim}\bfD^b(\mod\C)$ by Lemma~\ref{le:derived}.
\end{proof}

\begin{theorem}\label{th:KInj}
  Let $\C$ be a small abelian category. Then the triangulated category
  $\bfK(\Inj\Ind\C)$ has small coproducts and is compactly generated.
  Moreover, the canonical functor $\bfK(\Inj\Ind\C)\to\bfD(\Ind\C)$
  induces a triangle equivalence
\begin{equation}\label{eq:compacts}
\bfK(\Inj\Ind\C)^c\stackrel{\sim}\lto\bfD^b(\C).
\end{equation}
\end{theorem}
\begin{proof}
  The inclusion $\Ind\C\to\Mod\C$
  identifies \[\Inj\Ind\C=\{I\in\Inj\C\mid\Hom(X,I)=0\text{ for all
  }X\in\eff\C \}.\] This follows from the localisation sequence in
  Theorem~\ref{th:coh},
  since \[\Ind\C=\{Y\in\Mod\C\mid\Hom(X,Y)=0=\Ext^1(X,Y)\text{ for all
  }X\in\Eff\C\}\] by \cite[III.3]{Ga1962}. Then
  Proposition~\ref{pr:min} implies that the inclusion
  $\bfK(\Ind\C)\to\bfK(\Mod\C)$ identifies
\[\bfK(\Inj\Ind\C)=\{I\in\bfK(\Inj\C)\mid\Hom(X,I[n])=0\text{
  for all }X\in\eff\C,\, n\in\bbZ \},\] where $\eff\C\subseteq
\bfD^b(\mod\C)$ is viewed as subcategory of $\bfK(\Inj\C)$ via the
equivalence \eqref{eq:mod}. Now apply Proposition~\ref{pr:local} and
use that $\bfK(\Inj\C)$ is compactly generated by
Lemma~\ref{le:mod}. Thus $\bfK(\Inj\Ind\C)$ is compactly generated,
and $\bfK(\Inj\Ind\C)^c$ identifies with $\bfD^b(\C)$ thanks to
Corollary~\ref{co:derived}.
\end{proof}

\begin{remark}
  The first theorem from the introduction (labelled $\a=\aleph_0$) is
  a consequence of Theorem~\ref{th:KInj}, since a Grothendieck abelian
  category $\A$ that is generated by the full subcategory $\fp\A$ of
  finitely presented objects is equivalent to $\Ind\fp\A$ via the
  functor $\Ind\fp\A\to\A$ induced by the inclusion $\fp\A\to\A$.
\end{remark}

\begin{corollary}\label{co:locseq}
  The inclusion $\Inj\Ind\C\to\Inj\C$ induces a functor
  \[\bfK(\Inj\Ind\C)\lto \bfK(\Inj\C)\] that admits a left and a
  right adjoint. The left adjoint makes the following diagram
  commutative.
\[\begin{tikzcd}[row sep=scriptsize]
  \bfD^b(\mod\C)\arrow[tail]{d} \arrow[twoheadrightarrow]{rr}
  &&\bfD^b(\C)\arrow[tail]{d}\\
  \bfK(\Inj\C) \arrow[twoheadrightarrow,yshift=-1.35ex]{rr}
  \arrow[twoheadrightarrow,yshift=1.35ex]{rr}&& \bfK(\Inj\Ind\C)
  \arrow[tail]{ll}
\end{tikzcd}\] 
\end{corollary}
\begin{proof}
  The left adjoint of the inclusion $F\colon
  \bfK(\Inj\Ind\C)\to\bfK(\Inj\C)$ exists by construction and
  restricts to $\bfD^b(\mod\C)\to \bfD^b(\C)$; see
  Proposition~\ref{pr:local}. Next observe that $F$ preserves
  coproducts since its essential image are the objects perpendicular
  to a set of compact objects. Thus $F$ admits a right adjoint by
  Brown representability.
\end{proof}

The following consequence of Theorem~\ref{th:KInj} is due to
\v{S}\v{t}ov\'i\v{c}ek; his proof is different and based on an
analysis of fp-injective modules.

\begin{corollary}[\cite{St2011}]
 \pushQED{\qed} Let $A$ be a coherent ring. Then $\bfK(\Inj A)$ is a
compactly generated triangulated category and the canonical functor
$\bfK(\Inj A)\to\bfD(\Mod A)$ induces  a triangle equivalence
\[\bfK(\Inj A)^c\stackrel{\sim}\lto\bfD^b(\mod A).\qedhere\]
\end{corollary}

\subsection*{Functoriality}

We discuss the functoriality of the assignment
$\C\mapsto\bfK(\Inj\Ind\C)$.

Fix an additive functor $f\colon \C\to\D$ between small additive
categories.  Then \[f^*\colon\Mod\D\lto\Mod \C,\quad X\mapsto X\comp
f\] admits a right adjoint $f_*$ and a left adjoint $f_!$ \cite[\S
5]{Grothendieck/Verdier:1972a}. Note that $f_!$ extends $f$, that is,
$f_!$ sends each representable functor $\Hom_\C(-,X)$ to
$\Hom_\D(-,f(X))$. Thus $f_!$ restricts to a functor
$\Ind\C\to\Ind\D$.

Now suppose that $f$ is an exact functor between abelian
categories. Then $f^*$ restricts to a functor $\Ind\D\to\Ind\C$, since
$\Ind\C=\Lex(\C^\op,\Ab)$ and $\Ind\D=\Lex(\D^\op,\Ab)$.  Thus
$(f_!,f^*)$ yields an adjoint pair of functors
$\Ind\C\rightleftarrows\Ind\D$ and $f_!$ is exact. Therefore $f^*$
restricts to a functor
\[\Inj\Ind\D\lto\Inj\Ind\C.\]

\begin{proposition}
An exact functor $f\colon \C\to\D$ induces a functor
\[F^*\colon\bfK(\Inj\Ind\D)\lto\bfK(\Inj\Ind\C)\] that admits a left
and a right adjoint. The left adjoint makes the following diagram
commutative.
\[\begin{tikzcd}[row sep=scriptsize]
  \bfD^b(\C)\arrow[tail]{d} \arrow[twoheadrightarrow]{rr}{\bfD^b(f)}
  &&\bfD^b(\D)\arrow[tail]{d}\\
  \bfK(\Inj\Ind\C) \arrow[yshift=-1.35ex]{rr}[swap]{F_*}
  \arrow[yshift=1.35ex]{rr}{F_!}&& \bfK(\Inj\Ind\D)
  \arrow{ll}[description]{F^*}
\end{tikzcd}\] 
\end{proposition} 
\begin{proof}
  In Corollary~\ref{co:locseq} the assertion has been established for
  the canonical functors
\[p\colon\mod\C\lto\C\quad\text{and}\quad q\colon\mod\D\lto\D.\]

Now consider the sequence $(f_!,f^*,f_*)$ of functors making the
following diagram commutative.
\[\begin{tikzcd}
  \mod\D\arrow[tail]{rr}&&
  \Mod\D\arrow{d}[description]{f^*}&&\Inj\Ind\D\arrow[tail]{ll}\arrow{d}\\
  \mod\C\arrow[tail]{rr}\arrow{u}&&
  \Mod\C\arrow[xshift=-1.35ex]{u}{f_!}\arrow[xshift=1.35ex]{u}[swap]{f_*}&&\Inj\Ind\C\arrow[tail]{ll}
\end{tikzcd}\] We extend this diagram to complexes and obtain the
following diagram.
\[\begin{tikzcd}
  \bfD^b(\mod\D)\arrow[tail]{rr}&&
  \bfD(\Mod\D)\arrow{d}[description]{f^*}\arrow[twoheadrightarrow,yshift=1.35ex]{rr}{Q_!}
  \arrow[twoheadrightarrow,yshift=-1.35ex]{rr}[swap]{Q_*}&&
  \bfK(\Inj\Ind\D)\arrow[tail]{ll}[description]{Q^*}\arrow{d}[description]{F^*}\\
  \bfD^b(\mod\C)\arrow[tail]{rr}\arrow{u}&&
  \bfD(\Mod\C)\arrow[xshift=-1.35ex]{u}{f_!}\arrow[xshift=1.35ex]{u}[swap]{\bfR f_*}
  \arrow[twoheadrightarrow,yshift=1.35ex]{rr}{P_!}
  \arrow[twoheadrightarrow,yshift=-1.35ex]{rr}[swap]{P_*}&&
  \bfK(\Inj\Ind\C)\arrow[tail]{ll}[description]{P^*}
  \arrow[xshift=-1.35ex]{u}{F_!}\arrow[xshift=1.35ex]{u}[swap]{F_*}
\end{tikzcd}\]
Here, $\bfR f_*$ denotes the right derived functor. Thus it remains to
describe the vertical functors on the right. In fact, $F^*$ is
determined by the identity $P^*F^*=f^*Q^*$. Now set
$F_!:=Q_! f_! P^*$ and $F_*:=Q_*\bfR f_* P^*$. Then we have
\begin{align*}
  \Hom(F_!,-)&=\Hom(Q_! f_! P^*,-)\\ 
             &=\Hom(-,P_*f^*Q^*)\\ &=\Hom(-,P_*P^*F^*)\\ 
             &=\Hom(-,F^*)
               \intertext{and similarly}
               \Hom(-,F_*)&=\Hom(F^*,-).
\end{align*}

It remains to show that $F_!$ restricts on compacts to $\bfD^b(f)$,
after identifying the full subcategory of compacts in
$\bfK(\Inj\Ind\C)$ with $\bfD^b(\C)$ via \eqref{eq:compacts}.  This
assertion holds for $P_!$, $Q_!$, and $f_!$. Then the identity
$F_!P_!=Q_! f_!$ yields the assertion for $F_!$, using that the
diagram
\[\begin{tikzcd}
  \bfK^b(\C)\arrow{d}[swap]{\bfK^b(f)}\arrow{rr}{\sim}&&
\bfD^b(\mod\C)\arrow{d}\arrow[twoheadrightarrow]{rr}{\bfD^b(p)}
  &&\bfD^b(\C)\arrow{d}{\bfD^b(f)}\\
  \bfK^b(\D)\arrow{rr}{\sim}&&\bfD^b(\mod\D)\arrow[twoheadrightarrow]{rr}{\bfD^b(q)}
  &&\bfD^b(\D) 
\end{tikzcd}\] is commutative.
\end{proof}

\section{Grothendieck abelian categories}\label{se:grothendieck}

In this section we generalise the results from the previous sections
to arbitrary Grothendieck abelian categories. This involves the
concepts of locally presentable abelian and well-generated
triangulated categories.

\subsection*{Locally presentable abelian categories}
Grothendieck abelian categories are well-known to be locally
presentable in the sense of Gabriel and Ulmer \cite{GU1971}. We recall
this concept, refering to \cite{AR1994, GU1971} for details and
unexplained terminology.

Let $\A$ be a cocomplete category and fix a regular cardinal $\a$. An
object $X$ in $\A$ is \emph{$\a$-presentable} if the representable
functor $\Hom_\A(X,-)$ preserves $\a$-filtered colimits.  We denote by
$\A^\a$ the full subcategory which is formed by all $\a$-presentable
objects. Observe that $\A^\a$ is closed under $\a$-small colimits in
$\A$.  The category $\A$ is called \emph{locally $\a$-presentable} if
$\A^\a$ is essentially small and each object is an $\a$-filtered
colimit of $\a$-presentable objects. Moreover, $\A$ is \emph{locally
  presentable} if it is locally $\b$-presentable for some cardinal
$\b$.  Note that we have for each locally presentable category $\A$ a
filtration $\A=\bigcup_\b\A^\b$ where $\b$ runs through all regular
cardinals.

Let $\C$ be a small additive category and fix a regular cardinal $\a$.
When $\C$ has $\a$-small colimits we
write \[\Ind_\a\C:=\Lex_\a(\C^\op,\Ab)\] for the category of left
exact functors $\C^\op\to\Ab$ preserving $\a$-small products. This
category is locally $\a$-presentable. Conversely, for any locally
$\a$-presentable additive category $\A$ the assignment
$X\mapsto\Hom_\A(-,X)|_{\A^\a}$ induces an equivalence
\[\A\stackrel{\sim}\lto\Ind_\a\A^\a.\]

\subsection*{Grothendieck abelian categories}

We begin with a discussion of the localisation theory for Grothendieck
abelian categories.

\begin{proposition}\label{pr:quotient}
  Let $\A$ be a Grothendieck abelian category and $\a$ a regular
  cardinal. Suppose that $\A$ is locally $\a$-presentable and that
  $\A^\a$ is abelian. For a localising subcategory   $\B\subseteq\A$ 
  such that $\B\cap\A^\a$ generates $\B$, the following holds:
\begin{enumerate}
\item $\B$ and $\A/\B$ are locally $\a$-presentable Grothendieck
  abelian categories.
\item $\B^\a=\B\cap\A^\a$ and the quotient functor $\A\to\A/\B$ induces
  an equivalence\[\A^\a/\B^\a\stackrel{\sim}\lto (\A/\B)^\a.\]
\item The inclusion $\B\to\A$ induces a localisation sequence.
\[
\begin{tikzcd}[row sep=scriptsize]
\B^\a\arrow[tail]{rr}\arrow[tail]{d}&&
\A^\a\arrow[twoheadrightarrow]{rr}\arrow[tail]{d}&&\A^\a/\B^\a \arrow[tail]{d}\\
\B \arrow[tail,yshift=0.75ex]{rr} &&\A  \arrow[twoheadrightarrow,yshift=-0.75ex]{ll}
\arrow[twoheadrightarrow,yshift=0.75ex]{rr} &&\A/\B  \arrow[tail,yshift=-0.75ex]{ll}
\end{tikzcd}
\]
\end{enumerate}
\end{proposition}
\begin{proof}
  For the case $\a=\aleph_0$, see Theorems~2.6 and 2.8 of
  \cite{Kr1997}.  The general case is analogous; it amounts to
  identifying the sequence $\B\rightarrowtail\A\twoheadrightarrow\A/\B$ with the sequence
  $\Ind_\a\B^\a\to \Ind_\a\A^\a\to \Ind_\a(\A^\a/\B^\a)$ which is
  induced by $\B^\a\rightarrowtail \A^\a\twoheadrightarrow \A^\a/\B^\a$.
\end{proof}

The Popesco--Gabriel theorem yields the following
well-known\footnote{References are \cite[p.~4]{GU1971} or
  \cite[9.11.3]{Grothendieck/Verdier:1972a}.} consequence.

\begin{corollary}\label{co:PG}
  Any Grothendieck abelian category is locally presentable. Moreover,
  there exists a regular cardinal $\a$ such that $\A^\a$ is abelian.
\end{corollary}
\begin{proof}
  Let $\A$ be a Grothendieck abelian category with generator $U$ and
  set $\Ga:=\End_\A(U)$. Then the functor
  $\Hom_\A(U,-)\colon\A\to\Mod\Ga$ is fully faithful and admits an
  exact left adjoint $-\otimes_\Ga U$; it is the unique colimit
  preserving functor sending $\Ga$ to $U$. This induces an equivalence
  $\Mod\Ga/\C\xto{\sim}\A$, where $\C\subseteq\Mod\Ga$ denotes the
  localising subcategory of objects annihilated by $-\otimes_\Ga U$;
  see \cite{PG1964}. Now choose $\a$ so that $\mod_\a\Ga$ is abelian
  (see Lemma~\ref{le:bound} below) and contains a generator of
  $\C$. Then apply Proposition~\ref{pr:quotient}.
\end{proof}

Let $\C$ be a small additive category and fix a regular cardinal
$\a$. We write \[\mod_\a\C:=(\Mod\C)^\a\qquad\text{and}
\qquad\proj_\a\C:=\Proj\C\cap\mod_\a\C.\] The next lemma shows that
$\mod_\a\C$ is abelian when $\a$ is sufficiently large.

\begin{lemma}\label{le:bound}
The following conditions are equivalent:
\begin{enumerate}
\item The kernel of each morphism in $\mod\C$ belongs to $\mod_\a\C$.
\item The category $\proj_\a\C$ has \emph{pseudo-kernels}, that is, for each
  morphism $Y\to Z$ there exists a morphism $X\to Y$ making the sequence
  $X\to Y\to Z$ exact.
\item The category $\mod_\a\C$ is abelian. 
\end{enumerate}
\end{lemma}
\begin{proof}
(1) $\Rightarrow$ (2): The objects in $\proj_\a\C$ are precisely the
direct summands of coproducts $Y=\coprod_{i\in I}\Hom_\C(-,Y_i)$ with
$\card I<\a$. Clearly, $Y$ is the filtered colimit of subobjects
$\coprod_{i\in J}\Hom_\C(-,Y_i)$ with $\card J<\aleph_0$. This
colimit is $\a$-small, and it follows that any morphism $Y\to Z$ in
$\proj_\a\C$ is an $\a$-small filtered colimit of morphisms $Y_\la\to
Z_\la$ in $\proj_{\aleph_0}\C\subseteq\mod\C$. Thus
\[\Ker (Y\to Z)=\colim_\la\Ker(Y_\la\to Z_\la)\]
belongs to $\mod_\a\C$. It remains to observe that each object in
$\mod_\a\C$ is the quotient of an object in $\proj_\a\C$.

(2) $\Rightarrow$ (3): This follows from a standard argument
\cite[III.2]{Au1971} since each object in $\mod_\a\C$ is the cokernel
of a morphism in $\proj_\a\C$.

(3) $\Rightarrow$ (1): Clear. 
\end{proof}

When $\C$ has $\a$-small colimits, then the Yoneda functor
$\C\to\mod_\a\C$ admits a left adjoint; it is the $\a$-small colimit
preserving functor $\mod_\a\C\to\C$ taking each representable functor
$\Hom_\A(-,X)$ to $X$. Let $\eff_\a\C$ denote the full subcategory of
$\mod_\a\C$ consisting of the objects annihilated by this left
adjoint, and set $\Eff_\a\C:=\Ind_\a\eff_\a\C$.

\begin{proposition}\label{pr:presentation}
Let $\C$ be a small abelian category with $\a$-small coproducts and suppose
that $\Ind_\a\C$ is Grothendieck abelian. Then the
  inclusion $\Ind_\a\C\to\Mod\C$ induces a localisation sequence of
  abelian categories  
\begin{equation}\label{eq:presentation}
\begin{tikzcd}  
\Eff_\a\C \arrow[tail,yshift=0.75ex]{rr} && \Mod\C
  \arrow[twoheadrightarrow,yshift=-0.75ex]{ll}
  \arrow[twoheadrightarrow,yshift=0.75ex]{rr} &&\Ind_\a\C
  \arrow[tail,yshift=-0.75ex]{ll}
\end{tikzcd}
\end{equation} 
which restricts to the localisation sequence
\begin{equation*}
\begin{tikzcd}
  \eff_\a\C \arrow[tail,yshift=0.75ex]{rr} && \mod_\a\C
  \arrow[twoheadrightarrow,yshift=-0.75ex]{ll}
  \arrow[twoheadrightarrow,yshift=0.75ex]{rr} &&\C.
  \arrow[tail,yshift=-0.75ex]{ll}
\end{tikzcd}
\end{equation*}
\end{proposition}
\begin{proof}
  The inclusion $\Ind_\a\C\to\Mod\C$ has a left adjoint; it is the
  colimit preserving functor which is the identity on the
  representable functors \cite[V.1]{Ga1962}. This left adjoint is
  exact by the Popesco--Gabriel theorem \cite{PG1964}, and it sends
  $\a$-presentable objects to $\a$-presentable objects, since the
  right adjoint preserves $\a$-filtered colimits. This yields the left
  adjoint of the Yoneda functor $\C\to\mod_\a\C$. The rest follows
  from Proposition~\ref{pr:quotient}.
\end{proof}

There is an interesting consequence which seems worth mentioning.

\begin{corollary}\label{co:abelian}
  Let $\A$ be a locally $\a$-presentable Grothendieck abelian category
  such that $\A^\a$ is abelian. Then 
\[\frac{\Mod\A^\a}{\Eff_\a\A^\a}\stackrel{\sim}\lto\A\]
and $\A^\b$ is abelian for every
  regular cardinal $\b\ge \a$.
\end{corollary}
\begin{proof}
  We have a quotient functor $Q\colon\Mod\A^\a\to\A$ by
  Proposition~\ref{pr:presentation}, and this yields the presentation
  of $\A$.  Now observe that $\mod_\b\A^\a$ is abelian for all
  $\b\ge\a$ by Lemma~\ref{le:bound}. Thus $Q$ restricts to an exact
  quotient functor of abelian categories $\mod_\b\A^\a\to\A^\b$ by
  Proposition~\ref{pr:quotient}.
\end{proof}

\subsection*{Well-generated triangulated categories}

The triangulated analogue of a Grothendieck abelian category is a
well-generated triangulated category in the sense of Neeman
\cite{Ne2001}.  Such triangulated categories admit small coproducts
and are $\a$-compactly generated for some regular cardinal $\a$. Here,
we collect their essential properties and refer to
\cite{Kr2001,Ne2001} for further details.

 Fix a triangulated category $\T$ and suppose that $\T$ has small
 coproducts. Recall that a full triangulated subcategory
 $\Sc\subseteq\T$ is \emph{localising} if $\Sc$ is closed under
 all coproducts. For a regular cardinal $\a$, a full triangulated
 subcategory $\Sc\subseteq\T$ is \emph{$\a$-localising} if it is
 closed under $\a$-small coproducts.  An object $X$ in $\T$ is called
 \emph{$\a$-small} if every morphism $X\to\coprod_{i\in I}Y_i$ in $\T$
 factors through $\coprod_{i\in J}Y_i$ for some subset $J\subseteq I$
 with $\card J<\a$.

 A triangulated category $\T$ with small coproducts is
 \emph{$\a$-compactly generated} if there is a full subcategory
 $\T^\a$ satisfying the following:
\begin{enumerate}
\item $\T^\a\subseteq\T$ is an essentially small $\a$-localising
  subcategory consisting of $\a$-small objects.
\item $\T$ admits no proper localising subcategory containing $\T^\a$.
\item Given a family $(X_i\to Y_i)_{i\in I}$ of morphisms in
$\T$ such that the induced map $\Hom_\T(C,X_i)\to\Hom_\T(C,Y_i)$ is surjective for all
$C\in\T^\a$ and $i\in I$,  the induced map $\Hom_\T(C,\coprod_i
X_i)\to\Hom_\T(C,\coprod_i Y_i)$ is surjective.
\end{enumerate}
Then $\T^\a$ is uniquely determined by (1)--(3) and the objects in
$\T^\a$ are called \emph{$\a$-compact}. Also, $\T$ is
$\b$-compactly generated for every regular cardinal $\b\ge\a$, and
$\T^\b$ is the smallest $\b$-localising subcategory of $\T$ containing
$\T^\a$. In particular, $\T=\bigcup_\b\T^\b$ where $\b$ runs through
all regular cardinals.

The $\aleph_0$-compactly generated triangulated categories are
precisely the usual compactly generated triangulated categories. 

The most important aspect of the theory is that well-generated
categories behave well under localisation; this is in complete analogy
to Grothendieck abelian categories.

\subsection*{Localisation theory for well-generated triangulated
  categories}

We recall the basic facts from the localisation theory for
well-generated triangulated categories. For further details, see
\cite{Kr2010,Ne2001}.

The following is the analogue of Proposition~\ref{pr:quotient} for
abelian categories.

\begin{proposition}\label{pr:triangle-quotient}
  Let $\T$ be a triangulated category and $\a$ a regular
  cardinal. Suppose that $\T$ is $\a$-compactly generated. For a
  localising subcategory $\Sc\subseteq\T$ such that $\Sc\cap\T^\a$
  generates $\Sc$, the following holds:
\begin{enumerate}
\item $\Sc$ and $\T/\Sc$ are $\a$-compactly generated triangulated categories.
\item $\Sc^\a=\Sc\cap\T^\a$ and the quotient functor $\T\to\T/\Sc$ induces
  (up to direct summands) a triangle
  equivalence $\T^\a/\Sc^\a\xto{\sim} (\T/\Sc)^\a$.
\item The inclusion $\Sc\to\T$ induces a localisation sequence.
\[
\begin{tikzcd}[row sep=scriptsize]
\Sc^\a\arrow[tail]{rr}\arrow[tail]{d}&&
\T^\a\arrow[twoheadrightarrow]{rr}\arrow[tail]{d}&&\T^\a/\Sc^\a \arrow[tail]{d}\\
\Sc \arrow[tail,yshift=0.75ex]{rr} &&\T  \arrow[twoheadrightarrow,yshift=-0.75ex]{ll}
\arrow[twoheadrightarrow,yshift=0.75ex]{rr} &&\T/\Sc  \arrow[tail,yshift=-0.75ex]{ll}
\end{tikzcd}
\]
\end{enumerate}
\end{proposition}
\begin{proof}
 See Theorem~4.4.9 in \cite{Ne2001}.
\end{proof}

The following generalises Proposition~\ref{pr:local}, which treats the case
$\a=\aleph_0$.

\begin{proposition}\label{pr:perp}
  Let $\T$ be an $\a$-compactly generated triangulated category and
  $\Sc\subseteq\T^\a$ a triangulated subcategory that is closed under
  $\a$-small coproducts. Then the triangulated
  category \[\Sc^\perp:=\{Y\in\T\mid \Hom_\T(X,Y)=0\text{ for all
  }X\in\Sc\}\] has small coproducts and is $\a$-compactly generated.
  Moreover, the left adjoint of the inclusion $\Sc^\perp\to\T$ induces
  (up to direct summands) an equivalence $\T^\a/\Sc\xto{\sim}
  (\Sc^\perp)^\a$.
\end{proposition}
\begin{proof}
  Let $\bar\Sc\subseteq\T$ denote the smallest localising subcategory
  containing $\Sc$. Then the assertion follows from
  Proposition~\ref{pr:triangle-quotient}, since $\bar\Sc\cap\T^\a=\Sc$
  and the composite $\Sc^\perp=\bar\Sc^\perp
  \rightarrowtail\T\twoheadrightarrow\T/\bar\Sc$ is an equivalence by
  \cite[Theorem~9.1.16]{Ne2001}.
\end{proof}

\subsection*{The derived category of a Grothendieck abelian category}

The derived category of a Grothendieck abelian category is known to be
a well-generated triangulated category \cite{Ne2001b}.  For the
derived categories of rings and schemes, one finds a discussion of
$\a$-compact objects in \cite{Mu2011}. Here, we give a description of
the full subcategory of $\a$-compacts for any Grothendieck abelian
category, provided the cardinal $\a$ is sufficiently large. This is
based on the following special case.

\begin{proposition}\label{pr:murfet}
  Let $\C$ be a small additive category and $\a>\aleph_0$ a regular
  cardinal such that $\mod_\a\C$ is abelian. Then the derived category
  $\bfD(\Mod\C)$ is $\a$-compactly generated and the inclusion
  $\mod_\a\C\to\Mod\C$ induces an equivalence
\[\bfD(\mod_\a\C)\stackrel{\sim}\lto\bfD(\Mod\C)^\a.\]
\end{proposition}
\begin{proof}
  First observe that $\mod_\a\C$ is an abelian category with enough
  projective objects. Indeed, any $\a$-small coproduct of
  representable functors belongs to $\mod_\a\C$ and any object in
  $\mod_\a\C$ is a quotient of such a projective object.

  Now identify $\bfD(\Mod\C)$ with the full subcategory of
  $\bfK(\Proj\C)$ consisting of the K-projective complexes
  \cite{Sp1988}. Because $\bfD(\Mod\C)$ is compactly generated, the
  subcategory $\bfD(\Mod\C)^\a$ identifies with the smallest
  $\a$-localising subcategory containing all perfect complexes. The
  latter equals the full subcategory of K-projectives in
  $\bfK(\proj_\a\C)$, and this in turn identifies with
  $\bfD(\mod_\a\C)$.
\end{proof} 

Let $\A$ be an abelian category and $\B\subseteq\A$ a Serre
subcategory. Define the full subcategory
\[\bfD_\B(\A):=\{X\in\bfD(\A)\mid H^n(X)\in\B\text{ for all
}n\in\bbZ\}\subseteq\bfD(\A)\] and note that the quotient
functor  $\A\to\A/\B$ induces a functor
\[\bfD(\A)/\bfD_\B(\A)\lto\bfD(\A/\B).\] 
We will need the following fact.
\begin{lemma}\label{le:relative}
  The functor $\bfD(\A)/\bfD_\B(\A)\to\bfD(\A/\B)$ is a triangle
  equivalence when the quotient functor $\A\to\A/\B$ admits a right
  adjoint.
\end{lemma}
\begin{proof}
  Let $s$ denote right adjoint of $q\colon\A\to\A/\B$. The
  composite \[\bfK(\A/\B)\xto{s}\bfK(\A) \twoheadrightarrow\bfD(\A)
  \twoheadrightarrow\bfD(\A)/\bfD_\B(\A)\] annihilates each acyclic
  complex since $qs\cong\Id_{\A/\B}$. Thus $s$ induces an exact
  functor $\bfD(\A/\B)\to \bfD(\A)/\bfD_\B(\A)$ such that $qsX\cong X$
  for $X$ in $\bfD(\A/\B)$. On the other hand, for any complex $Y$ in
  $\bfD(\A)$ the cone of the adjunction morphism $Y\to sqY$ belongs to
  $\bfD_\B(\A)$; thus $Y\cong sqY$ in $\bfD(\A)/\bfD_\B(\A)$.
\end{proof}

The following result describes for any Grothendieck abelian category
the subcategory of $\a$-compact objects, provided the cardinal $\a$ is
sufficiently large.

\begin{theorem}\label{th:derived}
  Let $\A$ be a Grothendieck abelian category and $\a>\aleph_0$ a regular
  cardinal.  Suppose that $\A^\a$ is abelian and generates $\A$. Then
  the derived category $\bfD(\A)$ is $\a$-compactly generated and the
  inclusion $\A^\a\to\A$ induces a triangle equivalence
  $\bfD(\A^\a)\xto{\sim}\bfD(\A)^\a$.
\end{theorem}
\begin{proof}
  We set $\C:=\A^\a$ and identify $\A\xto{\sim}\Ind_\a\C$. Consider
  the following commutative diagram which is obtained from the pair of
  localisation sequences in Proposition~\ref{pr:presentation} by
  forming derived categories and using Lemma~\ref{le:relative}.
\[
\begin{tikzcd}[row sep=scriptsize]
 \bfD_{{\eff_\a}\C}(\mod_\a\C) \arrow[tail]{rr}\arrow[tail]{d} && \bfD(\mod_\a\C)
    \arrow[twoheadrightarrow]{rr}\arrow[tail]{d} &&\bfD(\C)\arrow[tail]{d}\\
\bfD_{{\Eff_\a}\C}(\Mod\C) \arrow[tail,yshift=0.75ex]{rr} && \bfD(\Mod\C)
  \arrow[twoheadrightarrow,yshift=-0.75ex]{ll}
  \arrow[twoheadrightarrow,yshift=0.75ex]{rr} &&\bfD(\Ind_\a\C)
  \arrow[tail,yshift=-0.75ex]{ll}
\end{tikzcd}\] The assertion follows from the localisation theory for
$\a$-compactly generated triangulated categories; see
Proposition~\ref{pr:triangle-quotient}. More precisely, we know from
Proposition~\ref{pr:murfet} that
\[\bfD(\Mod\C)^\a=\bfD(\mod_\a\C),\] and we
need to show that $\bfD_{{\eff_\a}\C}(\mod_\a\C)$ generates the
localising subcategory $\bfD_{{\Eff_\a}\C}(\Mod\C)$ of
$\bfD(\Mod\C)$. To see this, set $\T:=\bfD(\Mod\C)$ and let
$Q\colon\Mod\C\to\A$ denote the exact left adjoint of the functor
sending $X\in\A$ to $\Hom_\A(-,X)|_\C$ from \eqref{eq:presentation}.
Consider the cohomological functor
\[H\colon \T\stackrel{H^*}\lto\Mod\C\stackrel{Q}\lto\A\] and observe
that $H$ restricts to $\T^\a\to\A^\a$, since $Q$ restricts to
$\mod_\a\C\to\A^\a$. The kernel of $H$ equals
$\bfD_{{\Eff_\a}\C}(\Mod\C)$, and it is generated by the homotopy
colimits of countable sequences of morphisms in $\T^\a$ annihilated by
$H$; see \cite[\S 7.5]{Kr2010}. It remains to note that these homotopy
colimits belong to $\bfD_{{\eff_\a}\C}(\mod_\a\C)$ since
$\a>\aleph_0$.
\end{proof}

\begin{corollary}
 \pushQED{\qed} For a Grothendieck abelian category $\A$, the filtration
$\A=\bigcup_\a\A^\a$ induces a filtration
\[\bfD(\A)=\bigcup_{\a\text{ regular}}\bfD(\A^\a).\qedhere\]
\end{corollary}

\subsection*{Homotopy categories of injectives}

In recent work of Neeman \cite{Ne2014} it is shown that for any
Grothendieck abelian category the homotopy category of injective
objects is well-generated. Here we are slightly more specific and
provide an analogue of Theorem~\ref{th:KInj} for uncountable regular
cardinals.

\begin{theorem}\label{th:KInj2}
  Let $\A$ be a Grothendieck abelian category and $\a>\aleph_0$ a
  regular cardinal.  Suppose that $\A^\a$ is abelian and generates
  $\A$. Then the following holds:
\begin{enumerate}
\item The category $\bfK(\Inj\A)$ has small coproducts and is
  $\a$-compactly generated.
\item The left adjoint of the inclusion $\bfK(\Inj\A)\to\bfK(\A)$ restricts
to a quotient functor $\bfK(\A^\a)\twoheadrightarrow \bfK(\Inj\A)^\a$.
\item The canonical functor  $\bfK(\Inj\A)\to\bfD(\A)$ restricts
to a quotient functor $\bfK(\Inj\A)^\a \twoheadrightarrow\bfD(\A^\a)$.
\end{enumerate}
\end{theorem}
\begin{proof}
  We adapt the proof of Theorem~\ref{th:KInj} using the description of
  $\A$ via Proposition~\ref{pr:presentation}. As before, set
  $\C:=\A^\a$ and identify $\A\xto{\sim}\Ind_\a\C$. Consider
  \[\eff_\a\C\subseteq\bfD(\mod_\a\C)=\bfD(\mod\C)^\a\] and let
  \[\Sc:=\Loc(\eff_\a\C)\subseteq\bfD(\Mod\C)\] denote the localising subcategory
  generated by $\eff_\a\C$. Then $\bfK(\Inj\A)$ identifies with
  $\Sc^\perp$ in $\bfD(\Mod\C)$, and it follows from
  Proposition~\ref{pr:perp} that $\bfK(\Inj\A)$ is $\a$-compactly
  generated.  Moreover,
  \[\bfD(\Mod\C)/\Sc\stackrel{\sim}\lto\bfK(\Inj\A)\qquad\text{and}\qquad
\bfD(\mod_\a\C)/\Sc^\a\stackrel{\sim}\lto\bfK(\Inj\A)^\a.\] 

Our assertions about $\bfK(\Inj\A)^\a$ follow by inspection of the
following commuting diagram.
\[\begin{tikzcd}[row sep=3.7ex, column sep=3.7ex] 
  \bfK(\mod_\a\C)
  \arrow[tail]{rd}\arrow[twoheadrightarrow]{rr}\arrow[twoheadrightarrow]{dd}
  &&\bfK(\C)\arrow[tail]{rd}\arrow[twoheadrightarrow]{dd}\arrow[equal]{rr}&&
  \bfK(\C) \arrow[tail]{rd}\arrow[twoheadrightarrow]{dd}\\
  & \bfK(\Mod\C) \arrow[crossing over,
  twoheadrightarrow,yshift=0.75ex]{rr} && \bfK(\A)
  \arrow[crossing over, tail,yshift=-0.75ex]{ll}\arrow[crossing over,equal]{rr}&&\bfK(\A)\\
  \bfD(\mod_\a\C)\arrow[twoheadrightarrow]{rr}\arrow[tail]{rd}&&\bfK(\Inj\A)^\a
  \arrow[tail]{rd}\arrow[twoheadrightarrow]{rr}&&\bfD(\A)^\a
  \arrow[tail]{rd}\\
  &\bfD(\Mod\C)\arrow[twoheadrightarrow,yshift=0.75ex]{rr}
  \arrow[crossing over,tail,xshift=0.75ex]{uu}\arrow[crossing
  over,twoheadleftarrow,xshift=-0.75ex]{uu} && \bfK(\Inj\A)
  \arrow[tail,yshift=-0.75ex]{ll}\arrow[crossing
  over,tail,xshift=0.75ex]{uu} \arrow[crossing
  over,twoheadleftarrow,xshift=-0.75ex]{uu}\arrow[twoheadrightarrow,yshift=0.75ex]{rr}&&
  \arrow[tail,yshift=-0.75ex]{ll}\arrow[tail,xshift=0.75ex]{uu}
  \arrow[twoheadleftarrow,xshift=-0.75ex]{uu}\bfD(\A)
\end{tikzcd}\] We omit details but explain the construction of the
diagram. The two bottom rows are obtained by localising 
\[\bfK(\Inj\C)=\bfD(\Mod\C)\] with respect to $\Loc(\eff_\a\C)$ and
$\bfD_{{\Eff_\a}\C}(\Mod\C)$, restricting the left adjoints to the
full subcategories of $\a$-compact objects, and keeping in mind
that \[\Loc(\eff_\a\C)\subseteq\bfD_{{\Eff_\a}\C}(\Mod\C).\] The two
top rows follow from Proposition~\ref{pr:presentation}.  The left
adjoint of the inclusion $\bfK(\Inj\A)\to\bfK(\A)$ is obtained by
taking the composite
\[\bfK(\A)\rightarrowtail\bfK(\Mod\C)\twoheadrightarrow\bfD(\Mod\C)
\twoheadrightarrow\bfK(\Inj\A).\qedhere\]
\end{proof}

\begin{corollary}
  The inclusion $\bfK(\Inj\A)\to\bfK(\A)$ admits a left adjoint.
\end{corollary}
\begin{proof}
  The proof of Theorem~\ref{th:KInj2} yields an explicit left adjoint;
  for other constructions see \cite[Example~5]{Kr2012} and
  \cite[Theorem~2.13]{Ne2014}.
\end{proof}

\subsection*{The stable derived category}

Following Buchweitz \cite{Bu1987} and Orlov \cite{Or2004}, we define
the \emph{stable derived category} of a Grothendieck abelian category
$\A$ as the full subcategory of acyclic complexes in $\bfK(\Inj\A)$.
This is precisely the definition given in \cite{Kr2005} for a locally
noetherian category, and we denote this category by $\bfS(\A)$.

\begin{corollary}
  Let $\A$ be a Grothendieck abelian category and $\a>\aleph_0$ a
  regular cardinal.  Suppose that $\A^\a$ is abelian and generates
  $\A$. Then the stable derived category $\bfS(\A)$ is $\a$-compactly
  generated and fits into the following localisation sequence:
\[\begin{tikzcd}  
\bfS(\A) \arrow[tail,yshift=0.75ex]{rr} && \bfK(\Inj\A)
  \arrow[twoheadrightarrow,yshift=-0.75ex]{ll}
  \arrow[twoheadrightarrow,yshift=0.75ex]{rr} &&\bfD(\A)
  \arrow[tail,yshift=-0.75ex]{ll}
\end{tikzcd}\]
Moreover, the left adjoints preserve $\a$-compactness.
\end{corollary}
\begin{proof}
  The canonical functor $\bfK(\Inj\A)\to\bfD(\A)$ is a functor between
  $\a$-compactly generated triangulated categories by
  Theorems~\ref{th:derived} and \ref{th:KInj2}; it determines the
  localisation sequence. In particular, its kernel is $\a$-compactly
  generated by Proposition~\ref{pr:triangle-quotient}.
\end{proof}

\subsection*{Acknowledgements}

I am grateful to Jan \v{S}\v{t}ov\'i\v{c}ek for sharing the preprint
\cite{St2011} containing a precursor of Proposition~\ref{pr:pac}.
Also, I wish to thank Amnon Neeman for carefully reading a preliminary
version; as always his comments were most helpful.

\end{document}